\documentclass[a4paper]{article}
\usepackage{graphicx}
\usepackage{onecolpceurws}

\usepackage{amsmath,amsthm}
\usepackage{amssymb}
\usepackage{hyperref,cite}

\theoremstyle{plain}
   \newtheorem{teor}{Theorem}[section]

   \newtheorem{prop}[teor]{Proposition}
   \newtheorem{lemma}[teor]{Lemma}

\theoremstyle{definition}

\title{Permutation patterns in genome rearrangement problems}

\author{
Giulio Cerbai \\ giuliocerbai14@gmail.com
\and
Luca Ferrari \\ luca.ferrari@unifi.it
}

\institution{Dipartimento di Matematica e Informatica "U. Dini",\\ viale Morgagni 65, University of Firenze, Firenze, Italy}

\begin{document}
\maketitle
\setcounter{page}{124}

\begin{abstract}
In the context of the genome rearrangement problem, we analyze two well known models,
namely the block transposition and the prefix block transposition models,
by exploiting the connection with the notion of permutation pattern.
More specifically,
for any $k$, we provide a characterization of the set of permutations having distance $\leq k$ from the identity
(which is known to be a permutation class) in terms of what we call \emph{generating permutations}
and we describe some properties of its basis, which allow to compute such a basis for small values of $k$.
\end{abstract}
\vskip 32pt

\section{Introduction}

One of the major trends in bioinformatics and biomathematics is the study of the genome rearrangement problem.
Roughly speaking, given a genome, one is interested in understanding how the genome can evolve into another genome.
To give a proper formalization, several models for rearranging a genome have been introduced,
each of which defines a series of allowed elementary operations to be performed on a genome in order to obtain an adjacent one.
For several models, it is possible to define a \emph{distance} between two genomes,
by counting the minimum number of elementary operations needed to transform one genome into the other.
The investigation of the main properties of such a distance becomes then a key point in understanding the main features of the model under consideration.

\bigskip

A common formalization of any such models consists of encoding a genome using a \emph{permutation} (in linear notation)
and describing an elementary operation as a \emph{combinatorial} operation on the entries of such a permutation.
Many genome rearrangement models have been studied under this general framework.
Among them, the following ones are very well known.

\begin{itemize}

\item The \emph{reversal} model consists of a single operation, defined as follows:
a new permutation is obtained from a given one by selecting a cluster of consecutive elements and reversing it.
More formally, given $\pi =\pi_1 \pi_2 \cdots \pi_n$, a reversal is performed by choosing $i<j<n$ and then forming the permutation
$\sigma =\pi_1 \cdots \pi_{i-1} \boxed{\pi_j \pi_{j-1}\cdots \pi_{i+1} \pi_i}\pi_{j+1}\cdots \pi_n$.
This model was introduced in \cite{WEHM}, then studied for instance in \cite{BaPe1,HP}.

\item A variant of the reversal model is the \emph{prefix reversal model},
which is a specialization of the previous one in which the reversal operation can only be performed on a prefix of the given permutation.
This is clearly an easier model to investigate, which is also known as \emph{pancake sorting} (see for instance \cite{GP}).

\item A very popular and studied model is the \emph{transposition model}, see \cite{BaPe2}.
Given a permutatation $\pi =\pi_1 \cdots \pi_n$, a \emph{transposition operation} consists of taking two adjacent clusters of consecutive elements
and interchanging their positions. Formally, one has to choose indices $1\leq i<j<k\leq n+1$, then form the permutation
$\sigma =\pi_1 \cdots \pi_{i-1}\boxed{\pi_j \pi_{j+1}\cdots \pi_{k-1}}\boxed{\pi_i \pi_{i+1}\cdots \pi_{j-1}} \pi_k \cdots \pi_n$.

\item As for the reversal, also for the transposition model there is a ``prefix variant".
In the \emph{prefix transposition model} the leftmost block of elements to interchange is a prefix of the permutation.
Sorting by prefix transposition is studied in \cite{DM}.

\end{itemize}



Independently from the chosen model,
there are some general questions that can be asked in order to gain a better understanding of its combinatorial properties.
First of all, the operations of the model often (but not always) allow to define a \emph{distance} $d$ between two permutations $\rho$ and $\sigma$,
as the minimum number of elementary operations needed to transform $\rho$ into $\sigma$. Moreover, when the operations are nice enough,
the above distance $d$ could even be \emph{left-invariant}, meaning that, given permutations $\pi ,\rho ,\sigma$ (of the same length),
$d(\pi ,\rho )=d(\sigma \pi ,\sigma \rho )$. As a consequence, choosing for instance $\sigma =\rho^{-1}$,
the problem of evaluating the distance $d(\pi ,\rho )$ reduces to that of sorting $\pi$ with the minimum number of elementary allowed operations.
Now, if $d$ is a left-invariant distance on the set $S_n$ of all permutations of the same length, define the \emph{$k$-ball} of $S_n$ to be the set
$B_k ^{(d)}(n)=\{ \rho \in S_n \; |\; d(\rho ,id_n )\leq k\}$, where $id_n$ is the identity permutation of length $n$.
The following questions are quite natural to ask:
\begin{itemize}
\item compute the diameter of $B_k ^{(d)}(n)$, i.e. the maximum distance between two permutations of $B_k ^{(d)}(n)$;
\item compute the diameter of $S_n$, i.e. the maximum distance between two permutations of $S_n$;
\item characterize the permutations of $\partial B_k ^{(d)}(n)$, i.e. the permutations of $B_k ^{(d)}(n)$ having maximum distance from the identity;
\item characterize the permutations of $\partial S_n$, i.e. the permutations of $S_n$ having maximum distance from the identity;
\item characterize and enumerate the permutations of $B_k ^{(d)}(n)$;
\item design sorting algorithms and study the related complexity issues.
\end{itemize}

In the literature there are several results, concerning several evolution models, which give some insight into the above problems.
Our work starts from the observation that, in many cases, the balls $B_k ^{(d)}(n)$ can be characterized in terms of \emph{pattern avoidance}.
Recall that, given two permutation $\sigma \in S_k$ and $\tau =\tau_1 \tau_2 \cdots \tau_n \in S_n$, with $k\leq n$,
we say that $\sigma$ is a \emph{pattern} of $\tau$ when there exist $1\leq i_1 <i_2 <\cdots <i_k \leq n$ such that
$\tau_{i_1}\tau_{i_2}\cdots \tau_{i_k}$ as a permutation is isomorphic to $\sigma$ (which means that
$\tau_{i_1}, \tau_{i_2}, \ldots , \tau_{i_k}$ are in the same relative order as the elements of $\sigma$).
This notion of pattern in permutation defines an obvious partial order, and the resulting poset is known as the \emph{permutation pattern poset}.
When $\sigma$ is not a pattern of $\tau$, we say that $\tau$ \emph{avoids} $\sigma$.
A down-set $I$ (also called a permutation class) of the permutation pattern poset can be described in terms of its minimal excluded permutations
(or, equivalently, the minimal elements of the complementary up-set): these permutations are called the \emph{basis} of $I$.
The idea of studying the balls $B_k ^{(d)}(n)$ in terms of pattern avoidance is not new.
As far as we know, the first model which has been investigated from this point of view is the (whole) tandem duplication-random loss model:
Bouvel and Rossin \cite{BR} have in fact shown that, in such a model, the ball $B_k ^{(d)}=\bigcup_{n\geq 0}B_k ^{(d)}(n)$ is a class of pattern avoiding permutations,
whose basis is the set of minimal permutations having $d$ descents (here minimal is intended in the permutation pattern order).
Subsequent works \cite{BoPe,BF,CGM} have been done concerning the enumeration of the basis permutations of such classes.
More recently, Homberger and Vatter \cite{HV} described an algorithm for the enumeration of any polynomial permutation class,
which can be fruitfully used for all the above mentioned distances, since the resulting balls are indeed polynomial classes.
However, their results do not allow to find information on the basis of the classes.

In the present work we try to enhance what have been obtained in \cite{HV} in two directions.
First, we aim at giving a structural characterization of the balls for some of the above distances,
thus complementing the results in \cite{HV}, which is more concerned with computational issues.
Second, we provide some insight on the properties of the bases of such balls, hoping to gain a better understanding of them.
We will be mainly concerned with the block transposition and the prefix block transposition models,
leaving the reversal models to a future paper.

\bigskip

Some of the results of the present work are contained in the MSc thesis of the first author \cite{C}.


\section{Block transposition}

Among the four above mentioned models, the block transposition one is probably the hardest to investigate.

Denoting with $td$ the transposition distance, the permutation class $B_k ^{(td)}$ can be described in terms of its \emph{generating permutations}.

A \emph{strip} of $\pi =\pi_1 \pi_2 \cdots \pi_n \in S_n$ is a maximal consecutive substring $\pi_i \cdots \pi_{i+k-1}$ such that,
for all $j=1,\ldots, i+k-2$, $\pi_{j+1}=\pi_j +1$.

A permutation $\pi$ is said to be \emph{plus irreducible} \cite{AS} when, for all $i=1,\ldots ,n-1$, $\pi_{i+1}\neq \pi_i +1$.
In other words, $\pi$ is a plus irreducible permutation when it does not have points that are adjacent both in positions and values, with values increasing.
Equivalently, a permutation is plus irreducible if and only if all of its strips have length $1$.

Any permutation $\pi$ can be associated with a plus irreducible permutation, denoted $red(\pi )$,
which is obtained by replacing each strip of $\pi$ with its minimum element, then suitably rescaling the resulting word.
For instance, if $\pi =435612789$, then $red(\pi )=32415$.
It is easy to observe that $red(\pi )\leq \pi$ in the permutation pattern order.
Moreover, for every permutation $\pi$, we have that $td(\pi )=td(red(\pi ))$.

Given $\pi \in S_n$, let $v_1 ,\ldots ,v_n$ be nonnegative integers.
The \emph{monotone inflation} of $\pi$ through $v=(v_1 ,\ldots ,v_n )$ is the permutation $\pi [v]=\pi [id_{v_1},\ldots ,id_{v_n}]$
obtained from $\pi$ by replacing each element $\pi_i$ of $\pi$ with the identity permutation $id_{v_i}$ of length $i$ suitably rescaled,
so to mantain the relative order of the elements of $\pi$.
So, for instance, if $\pi =41352$ and $v=(0,2,1,3,2)$, we have $\pi [v]=\underbrace{\ldots}_4\underbrace{12}_1\underbrace{5}_3\underbrace{678}_5\underbrace{34}_2$.
In the following we will denote with $MI(\pi )$ the set of all monotone inflations of a permutation $\pi$
and with $MI(C)$ the set $\bigcup_{\pi \in C}MI(\pi )$, for a given set of permutations $C$.
The notion of monotone inflation is clearly related to that of geometric grid class \cite{AABRV}.
More specifically, given a $\{ -1,0,1\}$-matrix $M$ and denoting with $Geom(M)$ the geometric grid class of permutations determined by $M$,
if $\pi$ is a permutation and $M_\pi$ is its permutation matrix, then it is not difficult to realize that:
\begin{enumerate}
\item $Geom(M_\pi )=Geom(M_{red(\pi )})$;
\item $MI(\pi )=Geom(M_\pi )$;
\item $MI(\pi )=MI(red(\pi ))$.
\end{enumerate}

Now define a permutation $\pi$ to be \emph{generating} for $B_k ^{(td)}=\bigcup_{n\geq 0} B_k ^{(td)}(n)$ when
it is a maximal plus irreducible permutation of $B_k ^{(td)}$.
The set of all generating permutations for $B_k ^{(td)}$ will be called the \emph{generating set} of $B_k ^{(td)}$.
We thus have the following fact, whose easy proof is omitted.
\begin{prop}
For every $k$, $B_k ^{(td)}=\bigcup \{ MI(\pi )\, |\, \pi \textnormal{ is generating for $B_k ^{(td)}$} \}$.
\end{prop}

A very natural description of the balls $B_k ^{(td)}$ is then provided by its generating set.
This is our first open problem.

\bigskip

\textbf{Open problem.}\quad \emph{Characterize the generating permutations of $B_k ^{(td)}$, for every $k$}.

\bigskip

For instance, it is easy to realize that $B_1 ^{(td)}=MI(1324)$.
In the following, we will provide a structural description of the generating permutations of $B_k ^{(td)}$ for a generic $k$.

Our approach is based on the observation that a generating permutation for $B_k ^{(td)}$ is a plus irreducible permutation,
and so it will be convenient to work inside the poset of plus irreducible permutations,
seen as a subposet of the classical permutation pattern poset
(notice that this is also a subposet of the poset of \emph{peg permutations}, as defined in \cite{HV}).
In passing, we remark that the enumeration of plus irreducible permutations is well known:
denoting with $f_n$ the number of plus irreducible permutations of length $n+1$, we have the recurrence relation
$$
f_n=nf_{n-1}+(n-1)f_{n-2},
$$
for $n\geq 2$.
With initial conditions $f_0 =f_1 =1$, we get for $(f_n )_{n\geq 0}$ the exponential generating function $\frac{e^{-x}}{(1-x)^2}$
and the closed form $f_n =\sum_{k=0}^{n}(-1)^k (n+1-k)\frac{n!}{k!}$. This is sequence A000255 in \cite{S}, see also \cite{AAB}.

\bigskip

Suppose $\pi =\pi_1 \pi_2 \cdots \pi_n$ is a plus irreducible permutation of length $n$ in the generating set of $B_k ^{(td)}$. Inflate $\pi$ by choosing three (not necessarily distinct) indices $1\leq i\leq j\leq k\leq n$
and replacing $\pi_i ,\pi_j$ and $\pi_k$ by strips of suitable lengths, as follows:
\begin{itemize}
\item if the three indices are all distinct, take strips of length 2;
\item if two of the indices are equal, take the associated strip of length 3;
\item if all indices are equal, take a strip of length 4.
\end{itemize}

If $I$ is the multiset of the selected indices, the resulting permutation will be denoted $\pi_I$.
Now observe that, in all of the above cases, there exists a unique block transposition $\tau_I$ that breaks all the new strips in such a way that, in the resulting permutation, each pair of adjacent elements of a new strip becomes either nonadjacent or adjacent in the reverse way; more specifically, $\tau_I$ is the transposition with indices $i+1,j+2,k+3$. We call $\tilde{\pi}_I$ the permutation obtained from $\tilde{\pi}_I$ by applying $\tau_I$.
As an example, consider the permutation $\pi =1324$, and the multiset of indices $I=\{ 2,2,4 \}$;
then we get $\pi_I =1345267$ and $\tilde{\pi}_I =1352647$.

The following lemma gives some basic properties of the above described construction that will be useful in the sequel. The proof is easy and so left to the reader.

\begin{lemma}\label{prel}
Let $\pi =\pi_1 \cdots \pi_n$ be a plus irreducible permutation of length $n$
and $I$ a multiset of indices of $\pi$ of cardinality 3.
Then $\tilde{\pi}_I$ is a plus irreducible permutation of length $n+3$;
moreover, if $\pi_1 =1$ and $\pi_n =n$, then $(\tilde{\pi}_I )_1 =1$ and $(\tilde{\pi}_I )_{n+3} =n+3$.
\end{lemma}

We are now ready to give an explicit description of the generating set of the ball $B_k ^{(td)}$.
Such a result will be preceded by a technical proposition (stated without proof) which gives a recipe
to recursively construct the set of permutations which are obtainable by means of a single block transposition.
In the proof of the next theorem we also need the definition of breakpoint,
which can be found for instance in \cite{FLRTV}.
Given a permutation $\pi =\pi_1 \cdots \pi_n$, a \emph{breakpoint} of $\pi$ is an integer
$i\in \{ 0,1,\ldots n\}$ such that $\pi_{i+1}\neq \pi_{i}+1$.
By convention, $0$ and $n$ are breakpoints whenever $\pi_1 \neq 1$ and $\pi_n \neq n$, respectively.

\begin{prop}\label{plusone}
Let $\mathcal{I}(n)$ be the set of all multisets of cardinality 3 of $\{ 1,2,\ldots n\}$.
For every plus irreducible permutation $\pi \in S_n$, denote with $MI(\pi )^{+1}$ the set of all permutations
which can be obtained with a single block transposition from any permutation of $MI(\pi )$. Then
$$
MI(\pi )^{+1}=\bigcup_{I\in \mathcal{I}(n)}MI(\tilde{\pi}_I ).
$$
\end{prop}

\begin{teor}\label{genblock}
For every $k\geq 1$, the generating set of $B_k ^{(td)}$ is the set of all plus irreducible permutations
of length $3k+1$ and having distance $k$ from the identity.
\end{teor}

\begin{proof}
We start by showing that there exists a finite number $N=N(k)$ of permutations
$\alpha^{(1)},\ldots ,\alpha^{(N)}$ which are plus irreducible, of length $3k+1$
and at distance $k$ from the identity, such that $B_k ^{(td)}=\bigcup_{j=1}^{N}MI(\alpha^{(j)})$.
We can proceed by induction on $k$.
When $k=1$, we have already observed that $B_1 ^{(td)}=MI(1324)$,
and 1324 is plus irreducible, has length 3+1=4 and has distance 1 from the identity 1234.
Now consider a permutation $\overline{\pi}\in B_{k+1} ^{(td)}\setminus B_k ^{(td)}$;
this means, in particular, that there is a permutation $\pi \in B_k ^{(td)}$ such that
$\overline{\pi}$ is obtained from $\pi$ by a single block transposition.
Thus, using the induction hypothesis, we can say that there exists a plus irreducible permutation $\alpha$
of length $3k+1$ and having distance $k$ from the identity such that $\overline{\pi}\in MI(\alpha )^{+1}$.
By Proposition \ref{plusone}, there exists $I\in \mathcal{I}(3k+1)$ such that
$\overline{\pi}\in MI(\tilde{\alpha}_I )$.
Notice that $\mathcal{I}(3k+1)$ is finite and that, by Lemma \ref{prel},
$\tilde{\alpha}_I$ is plus irreducible and has length $3k+4$;
so what remains to prove is that $\tilde{\alpha}_I$ has distance $k+1$ from the identity.
Clearly $td(\tilde{\alpha}_I )\leq k+1$.
On the other hand, since Lemma \ref{prel} implies that $\tilde{\alpha}_I$ starts with 1 and ends with $3k+4$,
recalling that $\tilde{\alpha}_I$ is plus irreducible,
we have that $\tilde{\alpha}_I$ has exactly $3k+3$ breakpoints,
since the only indices that are not breakpoints are $0$ and $3k+4$.
Therefore, denoting with $Br(\pi )$ the number of breakpoints of $\pi$,
since $td(\pi )\geq \left\lceil \frac{Br(\pi )}{3}\right\rceil$ (this follows from an observation in \cite{BaPe2}),
we have that
$$
td(\tilde{\alpha}_I )\geq \left\lceil \frac{Br(\tilde{\alpha}_I )}{3}\right\rceil =\left\lceil \frac{3k+3}{3}\right\rceil=k+1,
$$
as desired.

To conclude the proof we now have to show that any plus irreducible permutation $\gamma$ of length $3k+1$ and
having distance $k$ from the identity is a generating permutation of $B_k ^{(td)}$.
In fact, since $\gamma \in B_k ^{(td)}$, $\gamma$ is the monotone inflation of
some generating permutation $\alpha$ of $B_k ^{(td)}$.
Therefore $\alpha$ is a plus irreducible permutation of length $3k+1$ at distance $k$ from the identity.
So in particular $\gamma$ and $\alpha$ have the same length, which means that necessarily $\gamma =\alpha$.
\end{proof}

The above theorem allows to design a procedure to list the generating set of $B_k ^{(td)}$:
starting from the identity of length $3k+1$, perform repeated monotone inflations as in Proposition \ref{plusone}
(for $k$ times) so to obtain all generating permutations of $B_k ^{(td)}$.
This is similar to the approach used in \cite{HV}.

For instance, when $k=2$, the generating set for $B_2 ^{(td)}$ consists of the eleven permutations
1324657, 1352647, 1354627, 1364257, 1426357, 1436527, 1462537, 1524637, 1536247, 1624357, 1632547.

Notice however that, in this way, it is possible to obtain the same generating permutation several times,
so in the list of permutations given in output by the above procedure one has to remove duplicates.
This is the main reason for which the described approach is not useful for enumerating the generating set.

\bigskip

\textbf{Open problem.}\quad \emph{Enumerate the generating permutations of $B_k ^{(td)}$, for every $k$}.

\bigskip

A very interesting information that we can get on $B_k ^{(td)}$ concerns its basis.
We start by recalling that monotone inflations are particular geometric grid classes \cite{AABRV};
as a consequence, the general theory of geometric grid classes allows us to say that
$B_k ^{(td)}$ is a permutation class having finite basis
(and also that it is \emph{strongly rational}, meaning that its generating function is rational,
together with the generating functions of all of its subclasses).
What we are able to do is to provide a nontrivial upper bound to the length of the basis elements,
which is clearly of great help in effectively computing the basis itself.

\begin{teor}\label{basis}
Every permutation belonging to the basis of $B_k ^{(td)}$ has length at most $3k+1$.
\end{teor}

\begin{proof}
We start by observing that, given $\pi$ basis permutation of $B_k ^{(td)}$, if $\pi$ were not plus irreducible,
then necessarily $red(\pi )<\pi$ and we have already observed that $td(\pi )=td (red(\pi ))$;
so $\pi$ would not be minimal among the permutations at distance $k$ from the identity, which is impossible.
Therefore we can assert that all basis permutations of $B_k ^{(td)}$ are plus irreducible.

Now it is easy to prove that every basis permutation has length at most $3k+2$.
Indeed, it is not difficult to show that a plus irreducible permutation of length $m$ contains as a pattern
at least one permutation of length $m-1$ that is plus irreducible as well.
Thus, if $\pi$ were a basis permutation having length greater than $3k+2$, then, in the poset of plus irreducible permutations,
there would exist at least one plus irreducible permutation $\sigma$ of length greater than $3k+1$ such that $\sigma <\pi$.
Since all generating permutations of $B_k ^{(td)}$ have length $3k+1$, $\sigma$ cannot belong to $B_k ^{(td)}$,
which is not possible since $\pi$ is a basis permutation.

Moreover, suppose that $\pi =\pi_1 \pi_2 \cdots \pi_n$ is a permutation in the basis of $B_k ^{(td)}$.
First of all we have that $\pi_1\neq 1$ and $\pi_n \neq n$,
since otherwise we could remove $\pi_1$ or $\pi_n$ thus obtaining a smaller permutation having the same distance from the identity,
against the minimality of $\pi$.
If $\pi$ had length $n=3k+2$, since $\pi$ is plus irreducible, there would exist $\gamma <\pi$ which is plus irreducible as well.
Since $\pi$ is minimal in the complement of $B_k ^{(td)}$, necessarily $td(\gamma )=k$.
This would imply that $\gamma=\gamma_1 \gamma_2 \cdots \gamma_{3k+1}$ is a generating permutation of $B_k ^{(td)}$.
This is however impossible, since it would be $\gamma_1 =1$ and $\gamma_{3k+1}=3k+1$ by Lemma \ref{prel} and the construction showed in Theorem \ref{genblock},
$\pi_1 \neq 1$ and $\pi_{3k+2} \neq 3k+2$ for what we have proved above,
and $\gamma$ is obtained from $\pi$ by removing a single entry.
\end{proof}

The above theorem also suggest a procedure to determine the basis of $B_k ^{(td)}$.
In the poset of plus irreducible permutations,
consider the set of permutations of length $3k+1$ which are not generating.
For each of them (say $\pi$), consider the set of permutation of length $3k$ covered by it:
if all of them are also below some generating permutation of $B_k ^{(td)}$,
then $\pi$ is in the basis of $B_k ^{(td)}$. Otherwise,
just repeat the same procedure starting from the permutations covered by $\pi$ which do not belong to $B_k ^{(td)}$.

As an instance, we have the following result.

\begin{prop}
The basis of $B_1 ^{(td)}$ is $\{ 321,2143,2413,3142\}$.
\end{prop}

\begin{proof}
Since $B_1 ^{(td)}=MI(1324)$, we perform the above procedure with all permutations of length 4
except $1324$.
A direct computation shows that the only permutations which cover only elements of $B_1 ^{(td)}$ are precisely
$2143,2413,3142$. Moreover, $321$ is the unique permutation of length 3 which is not in $B_k ^{(td)}$,
and all of its coverings are in $B_k ^{(td)}$, so $321$ is in the basis as well.
\end{proof}

\section{Prefix transposition}

If we restrict the block transposition operation to pairs of blocks such that
the first one is a prefix of the permutation, we obtain the so-called \emph{prefix transposition model}.
It is clearly a special case of the block transposition model and, as such, it is simpler to analyze.
Denoting with $ptd$ the prefix transposition distance, our first goal is to characterize the balls $B_k ^{(ptd)}$
in terms of generating permutations.
As a first example, it is easy to see that $B_1 ^{(ptd)}=MI(213)$.
The approach we use to determine the generating set $B_k ^{(ptd)}$ is slightly different
from the one we have used for the block transposition model.
In the present case, we are able to give an explicit construction of the generating set of $B_{k+1}^{(ptd)}$
starting from the generating set of $B_k ^{(ptd)}$.

\begin{prop}\label{prefixinfl}
Let $\tau \in S_n$ be a generating permutation of $B_k ^{(ptd)}$.
\begin{enumerate}

\item Suppose that $\tau =\pi a \rho b \gamma$, where $a<b\leq n$ and $\pi ,\rho$ and $\gamma$ are the subwords
of $\tau$ determined by such a decomposition. Then the permutation $(a+1)\hat{\rho}(b+1)\hat{\pi}a(b+2)\hat{\gamma}$ is a generating permutation of $B_{k+1}^{(ptd)}$,
where $\hat{\rho}, \hat{\pi}, \hat{\gamma}$ are obtained from $\rho ,\pi, \gamma$ (respectively) by increasing by 1
all the entries between $a$ and $b$ and by increasing by 2 all the entries that are greater than $b$.

\item Suppose that $\tau =\pi a \rho b \gamma$, where $n\geq a>b$ and $\pi ,\rho$ and $\gamma$ are the subwords
of $\tau$ determined by such a decomposition. Then the permutation $(a+2)\hat{\rho}b\hat{\pi}(a+1)(b+1)\hat{\gamma}$ is a generating permutation of $B_{k+1}^{(ptd)}$,
where $\hat{\rho}, \hat{\pi}, \hat{\gamma}$ are obtained from $\rho ,\pi, \gamma$ (respectively) by increasing by 1
all the entries between $b$ and $a$ and by increasing by 2 all the entries that are greater than $a$.

\item Suppose that $\tau =\pi a \rho$, where $a\leq n$ and $\pi$ and $\rho$ are the subwords
of $\tau$ determined by such a decomposition. Then the permutation $(a+1)\hat{\pi}a(a+2)\hat{\rho}$ is a generating permutation of $B_{k+1}^{(ptd)}$,
where $\hat{\pi}, \hat{\rho}$ are obtained from $\pi, \rho$ by increasing by 2 all the entries that are greater than $a$.

\end{enumerate}
\end{prop}

\begin{proof} We will give details only for the first case, the remaining two being analogous.
Since the prefix transposition model is a special case of the block transposition one, we have that,
if $\tau$ is a generating permutation for $B_k ^{(ptd)}$, we can construct generating permutations for
$B_{k+1}^{(ptd)}$ by suitably choosing two elements $a$ and $b$ of $\tau$ (possibly the same one),
then suitably inflating them and performing the prefix transposition operation which exchanges the prefix block ending with $a$
with the adjacent block ending with $b$. This is done in analogy with the construction described before Lemma \ref{prel}.

If $a<b$ and $a$ precedes $b$ in $\tau$, then we can decompose $\tau$ as $\tau =\pi a \rho b \gamma$.
After inflating $a$ and $b$, we thus get the permutation $\hat{\pi}a(a+1)\hat{\rho}(b+1)(b+2)\hat{\gamma}$,
where the elements of $\pi , \rho$ and $\gamma$ have been renamed, namely all entries greater than $a$ and smaller than $b$ have been
increased by 1 and all entries greater than $b$ have been increased by 2.
We can now perform the desired prefix transposition, which exchanges the prefix block $\hat{\pi}a$ with the adjacent block $(a+1)\hat{\rho}(b+1)$,
thus obtaining the predicted permutation.
\end{proof}

The above proposition gives a recipe for constructing generating permutations of $B_{k+1} ^{(ptd)}$ starting from those of $B_k ^{(ptd)}$.
Notice that, if $\tau$ has length $m$, then the permutations obtained with the previous construction have length $m+2$.
Since we have seen that $B_1 ^{(ptd)}=MI(213)$, a simple inductive argument shows that
the generating permutations of $B_k ^{(ptd)}$ we have produced all have length $2k+1$.
Actually, we have something stronger, which is the analogous of Theorem \ref{genblock} in the case of the prefix transposition model.
Since the proof is similar, we just give the statement.

\begin{teor}
For every $k\geq 1$, the generating set of $B_k ^{(ptd)}$ is the set of all plus irreducible permutations
of length $2k+1$ and having distance $k$ from the identity.
\end{teor}

However, in this case we can also enumerate the generating sets.

\begin{teor}
The generating set of $B_k ^{(ptd)}$ has cardinality $(2k)!/2^k$.
\end{teor}

\begin{proof} We observe that, if $\sigma =\sigma_1 \sigma_2 \cdots \sigma_{2k+3}$ is a generating permutation for $B_{k+1}^{(ptd)}$,
then it has been obtained from a generating permutation of $B_k ^{(ptd)}$ by one of the construction described in Proposition \ref{prefixinfl}.
However, $\sigma$ cannot be obtained in two different ways.
This can be shown by considering the elements $\sigma_1$ and $\sigma_1 -1$
(notice that, in this model, a generating permutation cannot start with 1).
\begin{enumerate}
\item If the element on the right of $\sigma_1 -1$ in $\sigma$ is larger than or equal to $\sigma_1 +2$,
then $\sigma$ is constructed as in 1 of Proposition \ref{prefixinfl}.
\item If the element on the right of $\sigma_1 -1$ in $\sigma$ is smaller than or equal to $\sigma_1 -2$,
then $\sigma$ is constructed as in 2 of Proposition \ref{prefixinfl}.
\item If the element on the right of $\sigma_1 -1$ in $\sigma$ is equal to $\sigma_1 +1$,
then $\sigma$ is constructed as in 3 of Proposition \ref{prefixinfl}.
\end{enumerate}

Since the three above cases are disjoint, we can conclude that $\sigma$ comes from a unique generating permutation of $B_k ^{(ptd)}$
through the construction of Proposition \ref{prefixinfl}.
Thus, the total number of generating permutations of $B_{k+1}^{(ptd)}$ is obtained by
multiplying the number of generating permutations of $B_k ^{(ptd)}$ by the number of possible inflations of each of them,
which is equal to the number of multisets of cardinality 2 of a set of cardinality $2k+1$, i.e. ${2k+2\choose 2}$.
Since the generating set of $B_1 ^{(ptd)}$ has cardinality $1={2\choose 2}$, a simple inductive argument shows that
the required cardinality is indeed equal to $\prod_{i=1}^{k}{2i\choose i}=(2k)!/2^k$.
\end{proof}

We have already observed that, for $k=1$, the generating set is $\{ 213\}$.
For $k=2$, the generating set is $\{ 32415, 41325, 31425, 24135, 24315,  42135\}$.

Concerning the basis of $B_k ^{(ptd)}$, we have been able to prove the analogue of Theorem \ref{basis},
however the proof is slightly more complicated, so we cannot reproduce it entirely here, due to limited space.

\begin{teor}
Every permutation belonging to the basis of $B_k ^{(ptd)}$ has length at most $2k+1$.
\end{teor}

\begin{proof} {\em (sketch).}\quad The fact that
the permutations in the basis of $B_k ^{(ptd)}$ must all have length at most $2k+2$
can be proved in a similar way as the first part of Theorem \ref{basis}.

Now suppose that $\pi =\pi_1 \pi_2 \cdots \pi_{2k+2}$ is a basis permutation for $B_k ^{(ptd)}$ of length $2k+2$,
and set $n=2k+2$.
Then it can be shown that $\pi$ has to be plus irreducible and that $\pi_{n}\neq n$,
i.e. the last element of $\pi$ is not its maximum.
From a previous observation, we know that it is possible to remove one element of $\pi$ in such a way that
the resulting permutation $\gamma= \gamma_1 \gamma_2 \cdots \gamma_{n-1}$ of length $n-1$ is plus irreducible.
However, since $\pi$ belongs to the basis of $B_k ^{(ptd)}$,
$\gamma$ has to be a generating permutation of $B_k ^{(ptd)}$.
Since it is possible to prove that the last element of any generating permutation of $B_k ^{(ptd)}$ is its maximum,
there are only two possibilities:
either the last element of $\pi$ is $n-1$ and $\gamma$ is obtained by removing $n$,
or the second-to-last element of $\pi$ is $n$ and $\gamma$ is obtained by removing the last element.

Since the two cases are symmetric in a well precise sense, we just consider the first one.
Our goal is now to show that we can remove another element from $\pi$ (different from $n$)
and obtain another plus irreducible permutation, which turns out to be a generating permutation:
this is however impossible, since it does not ends with its maximum.
In many cases, if we remove the last element $n-1$ of $\pi$, we do obtain a plus irreducible permutation.
The only cases in which this does not work are those in which $n-2$ is immediately before $n$ in $\pi$.
In such cases, if we remove $n-2$, we indeed get a plus irreducible permutation,
unless $n-3$ is immediately before $n-1$ in $\pi$.
By repeating this argument, we find that we are always able to remove an element different from $n$
and obtain a plus irreducible permutation, except for the permutation $\sigma= 2468\cdots n1357\cdots (n-1)$
(recall that $n=2k+2$, so $n$ is even).
Also in this last case, however, we can remove 1 from $\sigma$
and the permutation thus obtained is easily seen to be plus irreducible.
\end{proof}

Thanks to the above bound, we are able also in this case to compute the basis for small values of $k$.

\begin{prop}

\begin{enumerate}
\item The basis of $B_1 ^{(ptd)}$ is $\{ 132,321\}$.
\item The basis of $B_2 ^{(ptd)}$ consists of three permutations of length 4, namely 1432, 2143, 4321,
and fifteen permutations of length 5, namely
13524, 14253, 24351, 25314, 25413, 35142, 35214, 35241, 41352, 42513, 42531, 43152, 51324, 52413, 53142.
\end{enumerate}

\end{prop}

\subsubsection*{Acknowledgements}

Both authors are members of the INdAM Research group GNCS; they are partially supported by INdAM - GNCS 2018 project ``Propriet\'a combinatorie e rilevamento di pattern in strutture discrete lineari e bidimensionali" and by a grant of the "Fondazione della Cassa di Risparmio di Firenze" for the project "Rilevamento di pattern: applicazioni a memorizzazione basata sul DNA, evoluzione del genoma, scelta sociale".


\end{document}